\documentclass{amsart}
\usepackage{mathrsfs}
\usepackage{CJK,fancyhdr,amscd}
\usepackage{anysize}
\usepackage{amsmath,amssymb,amsfonts}
\usepackage{epsfig}
\usepackage{amsthm}
\usepackage{latexsym}
\usepackage{hyperref}
\usepackage{indentfirst, latexsym, bm}
\usepackage{graphics}
\usepackage{bbding}
\usepackage{dsfont}
\numberwithin{equation}{section}

\newtheorem{theo} {Theorem} [section]
\newtheorem{prop}[theo]{Proposition}
\newtheorem{coro}  [theo]     {Corollary}
\newtheorem{lemm}  [theo]     {Lemma}

\newtheorem{rema}  [theo]     {Remark}
\newtheorem{defi}  [theo]     {Definition}
\newtheorem{Claim}  [theo]     {Claim}

\newcommand{\rank}{\textmd{rank}}
\newcommand{\End}{\textmd{End}}

\begin{document}
\pagestyle{fancy}
\renewcommand{\headrulewidth}{0pt}

\fancyhf{}

\fancyhead[OL]{\leftmark}

\fancyhead[ER]{Cartan Formula and Its Applications}

\fancyhead[OR,EL]{$\cdot$\ \thepage\ $\cdot$}

\renewcommand{\sectionname}{}
\renewcommand{\sectionmark}[1]{\markboth{}{\thesection\ #1}}
\renewcommand{\sectionmark}[1]{\markright{\thesection\ #1}{}}

\title{Remarks on the Cartan Formula and Its Applications}

\author{Kefeng Liu}
\address{Department of Mathematics, University of California at Los Angeles,
 Angeles, CA 90095-1555, USA, Center of Mathematical Sciences, Zhejiang
University, Hangzhou, China}
\email{liu@math.ucla.edu,liu@cms.zju.edu.cn}


\author{Sheng Rao}
\address{Center of Mathematical Sciences, Zhejiang University, Hangzhou 310027, China}
\email{rsxiaotang@cms.zju.edu.cn}


%    General info
\subjclass[2000]{}

\date{July 28, 2010 and, in revised form, March 23, 2011.}



\begin{abstract}
In this short note, we present certain generalized versions of the
commutator formulas of some natural operators on manifolds, and give
some applications.

\end{abstract}

\maketitle

%\tableofcontents

\vspace{1cm}
\section{Introduction}

The purpose of this note is to present several general commutator
formulas of certain natural operators on Riemannian manifolds,
complex manifolds and generalized complex manifolds. We would like
to point out that such commutator formulas are essentially
consequences of the classical Cartan formula for Lie derivative, but
they have deep applications in geometry such as in studying the
smoothness of deformation spaces of manifolds. For example one
direct consequence of the commutator formula is the Tian-Todorov
lemma which is essential for proving the smoothness of the
deformation space of Calabi-Yau manifolds in \cite{[T],[T1]} and
also \cite{[T2]}. The general commutator formulas derived in the
note also have applications in proving smoothness of more general
deformation spaces such as that of the generalized complex manifolds
in \cite{[Li]}. We will discuss the applications of these commutator
formulas in deformation theory in our subsequent work.

\vspace{1cm}
\section{Cartan formula and a general commutator formula}
In this section we first present a general commutator formula on a
Riemannian manifold. We first fix notations. Given a smooth vector
field $X$, a smooth vector bundle $V$ on the smooth Riemannian
manifold $M$, and a connection $\nabla$ on $V$ which extends to
covariant derivative on the space of smooth $V$-valued differential
forms $\Omega^*(V)$, we denote by $L_X$ the Lie derivative acting on
$\Omega^*(V)$ and by $\lrcorner$ the contraction operator. We will
also denote by $\iota_X$ the contraction of a differential form by
the vector field $X$. Unless specially designated, $[\cdot,\cdot]$
will always denote the usual Lie bracket.

Our starting point is the following formula as Cartan observed:
\begin{equation}\label{C}
L_X\omega=X\lrcorner(\nabla\omega)+\nabla(X\lrcorner\omega),\
\text{for $\omega\in\Omega^*(V)$}.
\end{equation}
Then our general commutator formula can be stated as follows.

\begin{lemm}
For two smooth vector field $X$ and $Y$ on M,
\begin{equation}\label{cf}
[X,Y]\lrcorner\omega=X\lrcorner\nabla(Y\lrcorner\omega)+\nabla(X\lrcorner(Y\lrcorner\omega))-Y\lrcorner(X\lrcorner\nabla\omega)-Y\lrcorner\nabla(X\lrcorner\omega).
\end{equation}
\end{lemm}
\begin{proof}
On one hand, it is obvious by Cartan's formula (\ref{C}) that
$$L_X(Y\lrcorner\omega)=X\lrcorner\nabla(Y\lrcorner\omega)+\nabla(X\lrcorner(Y\lrcorner\omega)).$$
And on the other hand,
$$L_X(Y\lrcorner\omega)=(L_XY)\lrcorner\omega+Y\lrcorner (L_X\omega)=[X,Y]\lrcorner\omega+Y\lrcorner(X\lrcorner\nabla\omega+\nabla(X\lrcorner\omega)),$$
where the last identity apply Cartan's formula (\ref{C}) again. Then
(\ref{cf}) follows from these two identities above.
\end{proof}

\begin{rema}\label{Stf}\footnote{This remark is essentially due to R. Friedman \cite{[F]}.} This formula can be considered as a slight generalization of the
well-known commutator formula of Lie derivatives acting on
differential forms:
\begin{equation}\label{Ci}
[L_X, \iota_Y] =\iota_{[X,Y]}.
\end{equation}
It is easy to see that the commutator formula (\ref{Ci}) is a
special case of our formula when $V$ is taken as a trivial bundle on
the manifold $M$. In fact, applying both sides of (\ref{cf}) to a
differential form $\tau\in\Omega^*(M)$, we easily get
\begin{equation}\label{St}
Y\lrcorner(X\lrcorner d\tau)=X\lrcorner
d(Y\lrcorner\tau)+d\big(X\lrcorner(Y\lrcorner\tau)\big)-Y\lrcorner
d(X\lrcorner\tau)-[X,Y]\lrcorner\tau,
\end{equation}
which can also be written as
\begin{equation}\label{St1}
[X,Y]\lrcorner\tau=L_X(Y\lrcorner\tau)-Y\lrcorner(L_X\tau),
\end{equation}
which is just the formula (\ref{Ci}). (See Formula LIE $5$ of
Proposition $5.3$ on pp. $140$ of \cite{[S]}.)  Furthermore, if
$\tau\in\Omega^1(M)$, then (\ref{St}) becomes our familiar identity
$$d\tau(X,Y)=X\big(\tau(Y)\big)-Y\big(\tau(X)\big)-\tau([X,Y]),$$
by the vanishing of $d\big(X\lrcorner(Y\lrcorner\tau)\big)$ in
(\ref{St}).
\end{rema}

\begin{rema} Another proof of this formula is to use formula for
covariant derivative. Without loss of generality, we just consider
the special case that $V$ is a line bundle. We denote by $\theta$
and $\tau$ the connection $1$-form (matrix) with respect to the
connection $\nabla$ and a form of degree $k$ on $M$ respectively. By
definition, the covariant derivative of $\tau\otimes
s\in\Omega^k(V)$ is given by
\begin{equation}\label{cd}
\nabla(\tau\otimes s)=d\tau\otimes s+(-1)^k\tau\wedge\nabla s,
\end{equation}
where $s$ is a smooth section of $V$. Firstly, we can easily check
that
\begin{equation}\label{cd1}
Y\lrcorner\big(X\lrcorner(\tau\wedge\theta)\big)=\big(X\lrcorner(Y\lrcorner\tau)\big)\wedge\theta-X\lrcorner\big((Y\lrcorner\tau)\wedge\theta\big)+Y\lrcorner\big((X\lrcorner\tau)\wedge\theta\big).
\end{equation}
Actually, it is easy to know via a direct calculation that
\begin{align*}
   LHS
 &=Y\lrcorner\big((X\lrcorner\tau)\wedge\theta+(-1)^k\tau\wedge(X\lrcorner\theta)\big)\\
 &=Y\lrcorner\big((X\lrcorner\tau)\wedge\theta\big)+(-1)^k\Big(X\lrcorner\big(\theta\wedge(Y\lrcorner\tau)\big)-(-1)\theta\wedge\big(X\lrcorner(Y\lrcorner\tau)\big)\Big)\\
 &=Y\lrcorner\big((X\lrcorner\tau)\wedge\theta\big)+(-1)^k(-1)^{k-1}X\lrcorner\big((Y\lrcorner\tau)\wedge\theta\big)+(-1)^{k+k-2}\big(X\lrcorner(Y\lrcorner\tau)\big)\wedge\theta\\
 &=RHS.
\end{align*}
Then, without loss of generality, assuming that $s$ is a smooth
local frame of the smooth line bundle $V$, by adding up the tensor
products (\ref{cd1})$\otimes (-1)^ks$ and (\ref{St})$\otimes s$, we
have reproved our formula (\ref{cf}) according to the formula of
covariant derivative (\ref{cd}).
\end{rema}

\vspace{1cm}
\section{Commutator formula on complex manifolds} In this section, we consider an $n$-dimensional complex manifold
${M}$ and a holomorphic vector bundle $V$ over it. As the
applications of our commutator formula, we derive a general
commutator identity for any $V$-valued $(n,*)$-form and also any
$V$-valued $(*,*)$-form on $M$, and as an easy consequence we derive
the Tian-Todorov lemma. Unless otherwise mentioned, $\nabla$ will
always denote the Chern connection of the Hermitian holomorphic
vector bundle $V$ throughout this section.

Following pp. $152$ of \cite{[MK]}, we first introduce some notations. As usual, we let $A^{p, q}(V):=A^{p, q}({M},V)$ be the
space of smooth $(p, q)$-forms with coefficients in $V$. If
$X=\sum_{i=1}^nX^i\partial_i$ and $Y=\sum_{i=1}^nY^i\partial_i$,
then
$$[X,Y]=\sum_{i,j}(X^i\partial_iY^j-Y^i\partial_iX^j)\partial_j.$$
For the generalization, let
$$\varphi^i=\frac{1}{p!}\sum\varphi^i_{\bar{j}_1,\cdots,\bar{j}_p}d\bar{z}^{j_1}\wedge\cdots\wedge d\bar{z}^{j_p}$$
and
$$\psi^i=\frac{1}{q!}\sum\psi^i_{\bar{k}_1,\cdots,\bar{k}_q}d\bar{z}^{k_1}\wedge\cdots\wedge d\bar{z}^{k_q}.$$

\begin{defi} \label{br} For $\varphi=\sum_{i}\varphi^i\otimes\partial_i$
and $\psi=\sum_{i}\psi^i\otimes\partial_i$, we define
$$[\varphi,\psi]=\sum_{i,j=1}^n\big(\varphi^i\wedge\partial_i\psi^j-(-1)^{pq}\psi^i\wedge\partial_i\varphi^j\big)\otimes\partial_j,$$
where
$$\partial_i\varphi^j=\frac{1}{p!}\sum\partial_i\varphi^j_{\bar{j}_1,\cdots,\bar{j}_p}d\bar{z}^{j_1}\wedge\cdots\wedge
d\bar{z}^{j_p}$$ and similarly for $\partial_i\psi^j$. In
particular, if $\varphi,\psi\in A^{0,1}({M},T^{1,0}_{{M}})$, then
$$[\varphi,\psi]=\sum_{i,j=1}^n(\varphi^i\wedge\partial_i\psi^j+\psi^i\wedge\partial_i\varphi^j)\otimes\partial_j.$$
\end{defi}

With the setup above, we have the following general commutator
formula for $V$-valued $(n,*)$-forms on the complex manifold ${M}$.

\begin{prop}\label{PgTT}
For any holomorphic vector bundle $V$, any $\omega\otimes s\in
A^{n,*}(V)$\footnote{Here $\omega\in A^{n,*}(M)$ and $s$ is a smooth
section of $V$. In the following , we will adopt this convention
without chance for confusion.} and any $\phi_i\in
A^{0,1}({M},T^{1,0}_{{M}})$, $i=1,2$, there holds
\begin{equation}\label{gTT}
[\phi_1,\phi_2]\lrcorner(\omega\otimes
s)=\big(\phi_1\lrcorner\partial(\phi_2\lrcorner\omega)\big)\otimes
s-\partial\big(\phi_2\lrcorner(\phi_1\lrcorner\omega)\big)\otimes
s+\big(\phi_2\lrcorner\partial(\phi_1\lrcorner\omega)\big)\otimes s,
\end{equation}
or equivalently,
\begin{equation}\label{TT}
[\phi_1,\phi_2]\lrcorner\omega=\phi_1\lrcorner\partial(\phi_2\lrcorner\omega)-\partial\big(\phi_2\lrcorner(\phi_1\lrcorner\omega)\big)+\phi_2\lrcorner\partial(\phi_1\lrcorner\omega).
\end{equation}
\end{prop}
\begin{proof}We first show the following identity
$$[X,Y]\lrcorner(\tau\otimes s)=\big(X\lrcorner d(Y\lrcorner\tau)\big)\otimes s+d\big(X\lrcorner(Y\lrcorner\tau)\big)\otimes s-\big(Y\lrcorner(X\lrcorner d\tau)\big)\otimes s-\big(Y\lrcorner d(X\lrcorner\tau)\big)\otimes
 s,\quad \text{for $\tau\otimes s\in A^{n,*}(V)$}.$$
 One way to approach this identity is a direct application of (\ref{St}), while here we
adopt a lengthy but more intrinsic proof. In fact, we have
\begin{align*}
  &[X,Y]\lrcorner(\tau\otimes s)\\
 =&X\lrcorner\nabla\big(Y\lrcorner(\tau\otimes s)\big)+\nabla\Big(X\lrcorner\big(Y\lrcorner(\tau\otimes s)\big)\Big)-Y\lrcorner\big(X\lrcorner\nabla(\tau\otimes s)\big)-Y\lrcorner\nabla\big(X\lrcorner(\tau\otimes s)\big)\\
 =&\big(X\lrcorner d(Y\lrcorner\tau)\big)\otimes s+(-1)^{n-1}\big(X\lrcorner(Y\lrcorner\tau)\big)\otimes \nabla s+d\big(X\lrcorner(Y\lrcorner\tau)\big)\otimes s+(-1)^{n}\big(X\lrcorner(Y\lrcorner\tau)\big)\otimes \nabla s\\
  &-\big(Y\lrcorner(X\lrcorner d\tau)\big)\otimes s-(-1)^{n}\big(Y\lrcorner(X\lrcorner\tau)\big)\otimes\nabla s-\big(Y\lrcorner d(X\lrcorner\tau)\big)\otimes s-(-1)^{n-1}\big(Y\lrcorner(X\lrcorner\tau)\big)\otimes \nabla s\\
 =&\big(X\lrcorner d(Y\lrcorner\tau)\big)\otimes s+d\big(X\lrcorner(Y\lrcorner\tau)\big)\otimes s-\big(Y\lrcorner(X\lrcorner d\tau)\big)\otimes s-\big(Y\lrcorner d(X\lrcorner\tau)\big)\otimes
 s,
\end{align*}
where the first equality applies our general commutator formula
(\ref{cf}).

Then, if we take $\tau$ as $\omega\wedge d\bar{z}^{k_1}\wedge
d\bar{z}^{k_2}$ and $X,Y$ as
$(\phi_1)_{\bar{k}_1},(\phi_2)_{\bar{k}_2}$, respectively, we can
conclude the proof of (\ref{gTT}). In fact, set
$$\phi_i=(\phi_i)_{\bar{k}}\otimes d\bar{z}^k,$$
where
$(\phi_i)_{\bar{k}}=\sum^n_{j=1}(\phi_i)_{\bar{k}}^j\frac{\partial}{\partial
z^j}$ is a vector field of type $(1,0)$. Then by definition it is
easy to check that
$$[\phi_1,\phi_2]=[(\phi_1)_{\bar{k}_1},(\phi_2)_{\bar{k}_2}]\otimes(d\bar{z}^{k_1}\wedge d\bar{z}^{k_2})$$
and thus
$$[\phi_1,\phi_2]\lrcorner\omega=\big([(\phi_1)_{\bar{k}_1},(\phi_2)_{\bar{k}_2}]\lrcorner\omega\big)\wedge
(d\bar{z}^{k_1}\wedge
d\bar{z}^{k_2})=[(\phi_1)_{\bar{k}_1},(\phi_2)_{\bar{k}_2}]\lrcorner(\omega\wedge
d\bar{z}^{k_1}\wedge d\bar{z}^{k_2}).$$ Hence, by applying our
formula (\ref{cf}) in the special case (i.e., the commutator formula
(\ref{St}) as pointed out in Remark \ref{Stf}) to the complex
setting, and taking $\tau=\omega\wedge d\bar{z}^{k_1}\wedge
d\bar{z}^{k_2}$ and only the components of $(n-1,2)$-type forms, one
has
\begin{equation}\label{GTT}
\begin{aligned}
   &[(\phi_1)_{\bar{k}_1},(\phi_2)_{\bar{k}_2}]\lrcorner(\omega\wedge d\bar{z}^{k_1}\wedge d\bar{z}^{k_2})\\
 =\quad &(\phi_1)_{\bar{k}_1}\lrcorner\partial\big((\phi_2)_{\bar{k}_2}\lrcorner(\omega\wedge d\bar{z}^{k_1}\wedge
 d\bar{z}^{k_2})\big)+\partial\Big((\phi_1)_{\bar{k}_1}\lrcorner\big((\phi_2)_{\bar{k}_2}\lrcorner(\omega\wedge d\bar{z}^{k_1}\wedge d\bar{z}^{k_2})\big)\Big)\\
 -&(\phi_2)_{\bar{k}_2}\lrcorner\partial\big((\phi_1)_{\bar{k}_1}\lrcorner(\omega\wedge d\bar{z}^{k_1}\wedge
 d\bar{z}^{k_2})\big) -(\phi_2)_{\bar{k}_2}\lrcorner\big((\phi_1)_{\bar{k}_1}\lrcorner\partial(\omega\wedge d\bar{z}^{k_1}\wedge d\bar{z}^{k_2})\big).
\end{aligned}
\end{equation}
It is not difficult to know via a simple calculation that
$$[\phi_1,\phi_2]\lrcorner\omega=\phi_1\lrcorner\partial(\phi_2\lrcorner\omega)-\partial\big(\phi_2\lrcorner(\phi_1\lrcorner\omega)\big)+\phi_2\lrcorner\partial(\phi_1\lrcorner\omega),$$
by the vanishing of the last term of right-hand side of (\ref{GTT}).

\end{proof}

\begin{rema}\label{RPgTT}
It is interesting to write down the following several useful
identities from the proof of (\ref{TT}) and (\ref{gTT}): for any
$\omega\otimes s\in A^{*,*}(V)$,
$$\phi_1\lrcorner\bar{\partial}(\phi_2\lrcorner\omega)-\bar{\partial}\big(\phi_2\lrcorner(\phi_1\lrcorner\omega)\big)+\phi_2\lrcorner\bar{\partial}(\phi_1\lrcorner\omega)-\phi_2\lrcorner(\phi_1\lrcorner\bar{\partial}\omega)=0,$$
or equivalently,
\begin{equation}\label{bar}
\big(\phi_1\lrcorner\bar{\partial}(\phi_2\lrcorner\omega)\big)\otimes
s-\bar{\partial}\big(\phi_2\lrcorner(\phi_1\lrcorner\omega)\big)\otimes
s+\big(\phi_2\lrcorner\bar{\partial}(\phi_1\lrcorner\omega)\big)\otimes
s-\big(\phi_2\lrcorner(\phi_1\lrcorner\bar{\partial}\omega)\big)\otimes
s=0,
\end{equation}
and the (more) complete general commutator formula
\begin{equation}\label{cgTT2}
[\phi_1,\phi_2]\lrcorner\omega=\phi_1\lrcorner\partial(\phi_2\lrcorner\omega)-\partial\big(\phi_2\lrcorner(\phi_1\lrcorner\omega)\big)
+\phi_2\lrcorner\partial(\phi_1\lrcorner\omega)-\phi_2\lrcorner(\phi_1\lrcorner\partial\omega),
\end{equation}
or equivalently,
\begin{equation}\label{cgTT}
[\phi_1,\phi_2]\lrcorner(\omega\otimes
s)=\big(\phi_1\lrcorner\partial(\phi_2\lrcorner\omega)\big)\otimes
s-\partial\big(\phi_2\lrcorner(\phi_1\lrcorner\omega)\big)\otimes
s+\big(\phi_2\lrcorner\partial(\phi_1\lrcorner\omega)\big)\otimes s
-\phi_2\lrcorner(\phi_1\lrcorner\partial\omega)\otimes s.
\end{equation}
Moreover, there holds
$$[\phi_1,\phi_2]\lrcorner\omega=\phi_1\lrcorner d(\phi_2\lrcorner\omega)-d\big(\phi_2\lrcorner(\phi_1\lrcorner\omega)\big)
+\phi_2\lrcorner
d(\phi_1\lrcorner\omega)-\phi_2\lrcorner(\phi_1\lrcorner d\omega).$$
or equivalently,
\begin{equation}\label{dTT}
[\phi_1,\phi_2]\lrcorner(\omega\otimes s)=\big(\phi_1\lrcorner
d(\phi_2\lrcorner\omega)\big)\otimes
s-d\big(\phi_2\lrcorner(\phi_1\lrcorner\omega)\big)\otimes
s+\big(\phi_2\lrcorner d(\phi_1\lrcorner\omega)\big)\otimes s
-\phi_2\lrcorner(\phi_1\lrcorner d\omega)\otimes s.
\end{equation}
\end{rema}

Based on the argument above, we obtain another general commutator
identity for Hermitian holomorphic vector bundle valued
$(n,*)$-forms on complex manifolds.

\begin{theo}\label{TTh}
For any Hermitian holomorphic vector bundle $V$, any $\eta\in
A^{n,*}(V)$ and any $\phi_i\in A^{0,1}({M},T^{1,0}_{{M}})$, $i=1,2$,
$$[\phi_1,\phi_2]\lrcorner\eta=\phi_1\lrcorner\nabla(\phi_2\lrcorner\eta)-\nabla\big(\phi_2\lrcorner(\phi_1\lrcorner\eta)\big)+\phi_2\lrcorner\nabla(\phi_1\lrcorner\eta)-\phi_2\lrcorner(\phi_1\lrcorner\nabla\eta),$$
where $\nabla$ is the Chern connection of the Hermitian holomorphic
vector bundle $V$.
\end{theo}
\begin{proof}
Let $r=\rank (V)$. Assume that $s=\{s_1,\cdots,s_r\}$ is a local
holomorphic frame of $V$ and that $h=(h_{i\bar{j}})=(h(s_i,s_j))$ is
the matrix of the metric of $h$ under $s$. Without loss of
generality, we can locally set
\begin{equation}\label{assum}
\eta=\sum_{i=1}^r\omega_i\otimes s_i\quad \text{and}\quad \nabla
s_i=\sum_{j=1}^r\theta_{ij}\otimes s_j,
\end{equation}
where $\omega_i\in A^{n,*}(M)$ and $\theta_{ij}=\sum_{k=1}^r\partial
h_{i\bar{k}}\cdot h^{\bar{k}j}$ is the connection $(1,0)$-form of
$\nabla$ with respect to $s$.

Now we proceed to our proof. Firstly, we note that for any two
functions $f$ and $g$ on the complex manifold ${M}$, letting
$f\omega_i$ substitute for $\omega$ in (\ref{TT}), one has
$$
[\phi_1,\phi_2]\lrcorner
f\omega_i=-\partial\big(\phi_2\lrcorner(\phi_1\lrcorner
f\omega_i)\big)+\phi_1\lrcorner\partial(\phi_2\lrcorner
f\omega_i)+\phi_2\lrcorner\partial(\phi_1\lrcorner f\omega_i).
$$
So by (\ref{TT}) we have
$$
 0=-\partial f\wedge\big(\phi_2\lrcorner (\phi_1\lrcorner
\omega_i)\big)+\phi_1\lrcorner\big(\partial f\wedge(\phi_2\lrcorner
\omega_i)\big)+\phi_2\lrcorner\big(\partial f\wedge(\phi_1\lrcorner
\omega_i)\big)
$$
and its equivalent form
\begin{equation}\label{TTf}
 0=-g\cdot\partial f\wedge\big(\phi_2\lrcorner (\phi_1\lrcorner
\omega_i)\big)\otimes s_j+\phi_1\lrcorner\big(g\cdot\partial
f\wedge(\phi_2\lrcorner \omega_i)\big)\otimes
s_j+\phi_2\lrcorner\big(g\cdot\partial f\wedge(\phi_1\lrcorner
\omega_i)\big)\otimes s_j.
\end{equation}

Next, submitting  $\partial h_{i\bar{k}}\cdot h^{\bar{k}j}$ into
(\ref{TTf}) as $g\cdot\partial f$ and taking sums over $k$ and then
over $j$, then we obtain
\begin{equation}\label{con}
 0=-\nabla s_i\wedge\big(\phi_2\lrcorner (\phi_1\lrcorner
\omega_i)\big)+\phi_1\lrcorner\big(\nabla s_i\wedge(\phi_2\lrcorner
\omega_i)\big)+\phi_2\lrcorner\big(\nabla s_i\wedge(\phi_1\lrcorner
\omega_i)\big)
\end{equation}
according to (\ref{assum}).

Then, combining  (\ref{con}), (\ref{gTT}) and (\ref{bar}) with
$\omega$ and $s$ replaced by $\omega_i$ and $s_i$ respectively,  and
summing over $i$, we can complete our proof according to the formula
(\ref{cd}) for covariant derivative and the assumption
(\ref{assum}).

\end{proof}

Actually, we can easily generalize Theorem \ref{TTh} above to any
$\eta\in A^{*,*}(V)$.

\begin{coro}\label{TT2}
For any Hermitian holomorphic vector bundle $V$, any $\eta\in
A^{*,*}(V)$ and any $\phi_i\in A^{0,1}({M},T^{1,0}_{{M}})$, $i=1,2$,
there holds
\begin{equation}\label{acTT}
[\phi_1,\phi_2]\lrcorner\eta=\phi_1\lrcorner\nabla(\phi_2\lrcorner\eta)-\nabla\big(\phi_2\lrcorner(\phi_1\lrcorner\eta)\big)+\phi_2\lrcorner\nabla(\phi_1\lrcorner\eta)-\phi_2\lrcorner(\phi_1\lrcorner\nabla\eta),
\end{equation}
where $\nabla$ is the Chern connection of the Hermitian holomorphic
vector bundle $V$.
\end{coro}
\begin{proof}
Based on the identities (\ref{cgTT}) and (\ref{dTT}), we can obtain
(\ref{acTT}) by the same computation as we use to prove Theorem
\ref{TTh} and the details are left to the readers.
\end{proof}

Next, following the paper of S. Barannikov and M. Kontsevich
\cite{[BK]}, we can present some reformulation of the above results
as follows. Let us fix a $(k,l)$-form $\omega\in A^{k,l}(M)$. It
induces a linear map\footnote{In our manuscript, this map and also
the following map $\phi\mapsto\phi\lrcorner\eta$ were mistaken as
two isomorphisms, which is kindly pointed out by the referee.}
$$A^{0,q}(M,\bigwedge^pT^{1,0}_{M})\longrightarrow A^{k-p,q+l}(M)\footnote{For convention, here we set $k\geq p$.}:\phi\mapsto\phi\lrcorner\omega.$$
We define a map $\Delta_\omega$ from $\mathbf{t}$ to $A^{*,*}(M)$ by
the formula
$$\Delta_\omega\phi:=\partial(\phi\lrcorner\omega).$$
Similarly, let us fix a $V$-valued $(k,l)$-form $\eta\in
A^{k,l}(V)$. It induces a linear map
$$A^{0,q}(M,\bigwedge^pT^{1,0}_{M})\longrightarrow A^{k-p,q+l}(V):\phi\mapsto\phi\lrcorner\eta.$$
Then, a map $\diamond_\eta$ from $\mathbf{t}$ to $A^{*,*}(V)$ is
defined as the formula
$$\diamond_\eta\phi:=\nabla(\phi\lrcorner\eta).$$
Here $\mathbf{t}$ is the differential graded Lie algebras given by
$$\mathbf{t}=\bigoplus_k\mathbf{t}^k,\quad \mathbf{t}^k=\bigoplus_{p+q-1=k}A^{0,q}(M,\bigwedge^pT^{1,0}_{M}),$$
endowed with the differential $\bar{\partial}$, and the bracket
coming from the cup-product on $\bar{\partial}$-forms and the
standard Schouten-Nijenhuys bracket on polyvector fields.

Then we can generalize and restate Proposition \ref{PgTT}, and
restate Identity (\ref{cgTT2}) and Corollary \ref{TT2} as follows.

\begin{prop} (1) For any $\omega\in
A^{n,*}(M)$ and any $\phi_i\in A^{0,q}(M,\bigwedge^pT^{1,0}_{M})$,
$i=1,2$, there holds
$$[\phi_1,\phi_2]\lrcorner\omega=-\Delta_\omega(\phi_1\wedge\phi_2)+\phi_2\lrcorner\Delta_\omega\phi_1+\phi_1\lrcorner\Delta_\omega\phi_2.$$

 (2) For any $\omega\in
A^{*,*}(M)$ and any $\phi_i\in A^{0,1}({M},T^{1,0}_{{M}})$, $i=1,2$,
we have
 $$[\phi_1,\phi_2]\lrcorner\omega=-\big(\Delta_\omega(\phi_1\wedge\phi_2)-\phi_2\lrcorner\Delta_\omega\phi_1-\phi_1\lrcorner\Delta_\omega\phi_2+(\phi_1\wedge\phi_2)\lrcorner\partial\big).$$

 (3) For any $\eta\in A^{*,*}(V)$
and any $\phi_i\in A^{0,1}({M},T^{1,0}_{{M}})$, $i=1,2$, one has
$$[\phi_1,\phi_2]\lrcorner\eta=-\big(\diamond_\eta(\phi_1\wedge\phi_2)-\phi_2\lrcorner\diamond_\eta\phi_1-\phi_1\lrcorner\diamond_\eta\phi_2+(\phi_1\wedge\phi_2)\lrcorner\nabla\big).$$
\end{prop}

Finally, for the reader's convenience we briefly recall how to
derive the original Tian-Todorov lemma from the above commutator
formulas.

\begin{lemm} Let $M$ be an $n$-dimensional complex manifold
with a non-vanishing holomorphic $n$-form $\omega_0$, which is given
in a local coordinate chart $(U; z^1,\cdots,z^n)$ by
$\omega_0\big|_U=dz^1\wedge\cdots\wedge dz^n$. Then:

a) \emph{(Lemma $3.1$ in \cite{[T]}, or also Section $2$ in
\cite{[F]})} For $\omega_i\in A^{n-1,1}(M)$, $i=1,2$,
\begin{equation}\label{Ti}
[\omega_1,\omega_2]=-\partial\big(\omega_2\lrcorner\imath^{-1}(\omega_1)\big)+\omega_1\wedge\sharp(\partial\omega_2)+\omega_2\wedge\sharp(\partial\omega_1),
\end{equation}
where $\imath:A^{0,q}(M,T^{1,0}_{M})\rightarrow A^{n-1,q}(M)$ is the
natural isomorphism by contraction with $\omega_0$ and $\sharp$
denotes the obvious map identifying the $(n,q)$-form
$\eta\wedge\omega_0$ with the $(0,q)$-form $\eta$ by $\omega_0$,
i.e., $\sharp(\eta\wedge\omega_0)=\eta$.

b) \emph{(Lemma $1.2.4$ in \cite{[T1]}, or also Lemma $64$ in
\cite{[T2]})} For $\phi_i\in A^{0,1}(M,T^{1,0}_{M})$, $i=1,2$, with
$\partial(\phi_i\lrcorner\omega_0)=0$,
 \begin{equation}\label{To}
[\phi_1,\phi_2]\lrcorner\omega_0=-\partial\big(\phi_2\lrcorner(\phi_1\lrcorner\omega_0)\big).
\end{equation}
\end{lemm}

Actually, both $(\ref{Ti})$ and $(\ref{To})$ can be achieved by
$(\ref{TT})$. In fact, for each $\omega_i$, we have some $\phi_i\in
A^{0,1}(M,T^{1,0}_{M})$ via $\omega_i=\omega_0\lrcorner\phi_i$. Then
$[\omega_1,\omega_2]=\omega_0\lrcorner[\phi_1,\phi_2]$. Here we need
a simple commutator rule, that is, for any $\omega\in A^{k,l}(M)$
and $\psi\in A^{0,q}(M,\wedge^pT^{1,0}_{M})$, one has
\begin{equation}\label{cr}
\omega\lrcorner\psi=(-1)^{q(k+l-p)}\psi\lrcorner\omega.
\end{equation}
So by the commutator rule (\ref{cr}), we have
$[\omega_1,\omega_2]=[\phi_1,\phi_2]\lrcorner\omega_0$ and
$$\omega_1\wedge\sharp(\partial\omega_2)
 =-\partial(\omega_0\lrcorner\phi_2)\lrcorner\phi_1
 =(-1)^{(n+1)+n-1}\phi_1\lrcorner\partial(\phi_2\lrcorner\omega_0)
 =\phi_1\lrcorner\partial(\phi_2\lrcorner\omega_0).
$$
Similarly,
$\omega_2\wedge\sharp(\partial\omega_1)=\phi_2\lrcorner\partial(\phi_1\lrcorner\omega_0)$.
It is easy to check that
$$-\partial\big(\omega_2\lrcorner\iota^{-1}(\omega_1)\big)=-\partial\big(\phi_1\lrcorner(\phi_2\lrcorner\omega_0)\big)=-\partial\big(\phi_2\lrcorner(\phi_1\lrcorner\omega_0)\big).$$
Therefore, we obtain an equivalent form of Tian's identity,
\begin{equation}\label{Ti1}
[\phi_1,\phi_2]\lrcorner\omega_0=-\partial\big(\phi_2\lrcorner(\phi_1\lrcorner\omega_0)\big)+\phi_1\lrcorner\partial(\phi_2\lrcorner\omega_0)+\phi_2\lrcorner\partial(\phi_1\lrcorner\omega_0),
\end{equation}
which is just the identity (\ref{TT}) with $\omega=\omega_0$.

As for Todorov's identity (\ref{To}), we just need notice that the
condition $\partial(\phi_i\lrcorner\omega_0)=0$ results in the
vanishing of the last two terms in the right-hand side of
(\ref{Ti1}).

By this crucial Tian-Todorov lemma, the well-know
$\partial\bar{\partial}$-lemma and Kuranishi's construction of power
series, Tian \cite{[T]} and Todorov \cite{[T1]} proved the famous
Bogomolov-Tian-Todorov unobstrution theorem. It can be stated rough
as follows. Let $M$ be a Calabi-Yau manifold, where $n=\dim M\geq
3$. Let $\pi:X\rightarrow S$, with central fiber $\pi^{-1}(0)=M$ be
the Kuranishi family of $M$, then the Kuranishi space $S$ is a
non-singular complex analytic space and $\dim S = \dim
H^1_{\mathds{C}}(M, \Theta_M) = \dim H^1_{\mathds{C}}(M,
\Omega^{n-1}),$ where $\Theta$ is the holomorphic tangent bundle of
$M$.

\vspace{1cm}
\section{Twisted commutator formula on generalized complex manifolds}
In this section, we prove a twisted commutator formula on
generalized complex manifolds, reprove Corollary \ref{TT2} for any
Hermitian holomorphic vector bundle and obtain a more general
commutator formula in Corollary \ref{main} as the applications of
our twisted commutator formula.

First of all, let us introduce some notations on generalized complex
geometry and we refer the readers to \cite{[G],[Li]} and the
references therein for a more detailed and systematic treatment of
generalized complex geometry. Here we just list some basic concepts
we need in this note.

Let $\check{M}$ be a smooth manifold, $T:=T_{\check{M}}$ the tangent
bundle of $\check{M}$ and $T^*:=T^*_{\check{M}}$ its cotangent
bundle. In the generalized complex geometry, for any $X,Y\in
C^\infty(T)$ and $\xi,\eta\in C^\infty(T^*)$, $T\oplus T^*$ is
endowed with a \emph{canonical} nondegenerate \emph{inner product}
given by
\begin{equation}\label{cinner}
\langle
X+\xi,Y+\eta\rangle=\frac{1}{2}\big(\iota_X(\eta)+\iota_Y(\xi)\big),
\end{equation}
 and there is an important canonical bracket on $T\oplus T^*$, so-called
\emph{Courant bracket}, which is defined by
\begin{equation}\label{Courant bracket}
[X+\xi,Y+\eta]=[X,Y]+L_X\eta-L_Y\xi-\frac{1}{2}d\big(\iota_X(\eta)-\iota_Y(\xi)\big).
\end{equation}
Here, $[\cdot,\cdot]$ on the right-hand side is the ordinary Lie
bracket of vector fields. Note that on vector fields the Courant
bracket reduces to the Lie bracket; in other words, if $pr_1:T\oplus
T^*\rightarrow T$ is the natural projection,
$$pr_1([A,B])=[pr_1(A),pr_1(B)]),$$
for any $A,B\in C^\infty(T\oplus T^*).$

A \emph{generalized almost complex structure} on $\check{M}$ is a
smooth section $J$ of the endomorphism bundle $\End(T\oplus T^*)$,
which satisfies both symplectic and complex conditions, i.e.
$J^*=-J$ (equivalently, orthogonal with respect to the canonical
inner product (\ref{cinner})) and $J^2=-1$. We can show that the
obstruction to the existence of a generalized almost complex
structure is the same as that for an almost complex structure. (See
Proposition 4.15 in \cite{[G]}.) Hence it is obvious that
(generalized) almost complex structures only exist on the
even-dimensional manifolds. Let $E\subset(T\oplus T^*)\otimes
\mathbb{C}$ be the $+i$-eigenbundle of the generalized almost
complex structure $J$. Then if $E$ is Courant involutive, i.e.
closed under the Courant bracket (\ref{Courant bracket}), we say
that $J$ is \emph{integrable} and also a \emph{generalized complex
structure}. Note that $E$ is a maximal isotropic subbundle of
$(T\oplus T^*)\otimes \mathbb{C}$.

As observed by P. \v{S}evera and A. Weinstein \cite{[SW]}, the
Courant bracket (\ref{Courant bracket}) on $T\oplus T^*$ can be
twisted by a real, closed $3$-form $H$ on $\check{M}$ in the
following way: given $H$ as above, define another important bracket
$[\cdot,\cdot]_H$ on $T\oplus T^*$ by
$$[X+\xi,Y+\eta]_H=[X+\xi,Y+\eta]+\iota_Y\iota_X (H),$$
which is called \emph{$H$-twisted Courant bracket}.
\begin{defi} \label{}
A generalized complex structure $J$ is said to be \emph{twisted
generalized complex} with respect to the closed $3$-form $H$ when
its $+i$-eigenbundle $E$ is involutive with respect to the
$H$-twisted Courant bracket and then the pair $(\check{M},J)$ is
called an \emph{$H$-twisted generalized complex manifold}.
\end{defi}

From now on, we consider the $H$-twisted generalized complex
manifold $(\check{M},J)$ defined as above. Postponing listing some
more notions in need, we must remark that they are not exactly the
same as the usual ones since we just define them for our
presentation below, and maybe miss their usual geometrical meaning.
The
 \emph{twisted de Rham differential} is given by
$$d_R=d+(-1)^k R\wedge\cdot,$$
where $R\in \Omega^k(\check{M},\mathbb{R})$. A natural action of
$T\oplus T^*$ on smooth differential forms is given by
$$(X+\xi)\cdot\alpha=\iota_X(\alpha)+\xi\wedge\alpha,\quad \text{for any
$X\in C^\infty(T),\ \xi\in C^\infty(T^*)$ and $\alpha\in
\Omega^*(\check{M},\mathbb{C})$}.$$
 Actually, this action can be
considered as 'lowest level' of a hierarchy of actions on the
bundles $T\bigoplus(\oplus_r\wedge^r T^*)$, $r=1,2,\cdots$, defined
by the similar formula
$$(X+\xi_1+\xi_2+\cdots)\cdot\alpha=\iota_X(\alpha)+\xi_1\wedge\alpha+\xi_2\wedge\alpha+\cdots,$$
for any $X\in C^\infty(T),\ \xi_1+\xi_2+\cdots\in
C^\infty(\oplus_r\wedge^r T^*)$ and $\alpha\in
\Omega^*(\check{M},\mathbb{C})$. Then in the following discussion we
adopt the action of $A=A_1\wedge\cdots\wedge A_k\in
C^\infty\Big(\bigwedge^k \big(T\bigoplus(\oplus_r\wedge^r
T^*)\big)\Big)$ on $\Omega^*(\check{M},\mathbb{C})$ given by
\begin{equation}\label{action}
A\cdot\alpha=(A_1\wedge\cdots\wedge A_k)\cdot\alpha\equiv A_1\cdot
A_2\cdot\cdots \cdot A_k\cdot\alpha,\qquad\text{for any $\alpha\in
\Omega^*(\check{M},\mathbb{C})$}.
\end{equation}
The \emph{generalized Schouten bracket} for $A=A_1\wedge\cdots\wedge
A_p\in C^\infty(\wedge^p (T\oplus T^*))$ and
$B=B_1\wedge\cdots\wedge B_q\in C^\infty(\wedge^q (T\oplus T^*))$ is
defined as
$$[A,B]_R=\sum_{i,j}(-1)^{i+j}[A_i,B_j]_R\wedge A_1\wedge\cdots\wedge \hat{A}_i\wedge
\cdots\wedge A_p\wedge
B_1\wedge\cdots\wedge\hat{B}_j\wedge\cdots\wedge B_q,$$ where
$\hat{}$ means 'omission', the \emph{$R$-twisted Courant bracket}
 $[A_i,B_j]_R$ is defined as $[A_i,B_j]+\iota_{Y_j}\iota_{X_i}(R)$ if
we take $A_i=X_i+\xi_i$ and $B_j=Y_j+\eta_j$, and the action of
$[A_i,B_j]_R$ comply with the principle of (\ref{action}). Here we
note that if $R$ is a $3$-form and $X+\xi$, $Y+\eta\in
C^\infty(T\oplus T^*)$, then the $R$-twisted Courant bracket
$[X+\xi,Y+\eta]_R$ still lie in $C^\infty(T\oplus T^*)$. However,
for $R$ being general, the bracket $[X+\xi,Y+\eta]_R$ doesn't lie in
$C^\infty(T\oplus T^*)$ in general since $\iota_Y\iota_X (R)$ is not
necessarily a $1$-form, but in $C^\infty(T\bigoplus(\oplus\wedge^*
T^*))$; hence this bracket still makes sense under the action
(\ref{action}).

\begin{prop}\label{Hodd}\emph{(See also Lemma $4.24$ of \cite{[G]}, ($17$) of
\cite{[KL]} and Lemma $2$ of \cite{[Li]}.)} For any smooth
differential form $\rho$, any smooth odd-degree form $R$ and any
$A\in C^\infty(\wedge^p E^*)$, $B\in C^\infty(\wedge^q E^*)$, we
have
\begin{equation}\label{Ht}
d_R(A\cdot B\cdot\rho)=(-1)^pA\cdot
d_R(B\cdot\rho)+(-1)^{(p-1)q}B\cdot
d_R(A\cdot\rho)+(-1)^{p-1}[A,B]_R\cdot\rho+(-1)^{p+q+1}A\cdot B\cdot
d_R\rho.
\end{equation}
\end{prop}
\begin{proof}
Firstly, we consider the initial case, i.e., $A,B\in C^\infty(E^*)$.
It is proved by Gualtieri in Lemma $4.24$ of \cite{[G]} that
\begin{equation}\label{initial}
A\cdot B\cdot d\rho=d(B\cdot A\cdot\rho)+B\cdot d(A\cdot\rho)-A\cdot
d(B\cdot\rho)+[A,B]\cdot\rho-d\langle A,B\rangle\wedge \rho,
\end{equation}
where $A,B\in C^\infty(T\oplus T^*)$. Actually, (\ref{initial}) is
essentially due to the commutator formula (\ref{Ci}) and classical
Cartan formula $L_X=d\circ\iota_X+\iota_X\circ d$. In our case, we
can drop the last term involving the inner product.

Later, Kapustin and Li proved the $H$-twisted version in ($17$) of
\cite{[KL]} and then Li generalized it to any $A\in
C^\infty(\wedge^p E^*)$ and $B\in C^\infty(\wedge^q E^*)$ in Lemma
$2$ of \cite{[Li]}, where $H$ is a real closed $3$-form\footnote{In
\cite{[Li]}, Yi Li proved an analog of the Bogomolov-Tian-Todorov
theorem for $H$-twisted generalized Calabi-Yau manifolds by his
critical Lemma $2$, that is, the unobtruction and smoothness of the
moduli space of generalized complex structures on a compact
$H$-twisted generalized Calabi-Yau manifold.}. Here we give a
slightly more general version when $R$ is any smooth form of odd
degree. For the reader's convenience, we will write down the details
as follows though the essential idea of this process is due to
\cite{[Li]}.

Now let us compute $(A\cdot B)\cdot (R\wedge\rho)$. Let $A=X+\xi$,
$B=Y+\eta$ and $R\in \Omega^k(\check{M},\mathbb{R})$ with odd $k$.
By a direct computation and the notations introduced above, we have
the following two equalities
\begin{equation}\label{BHA}
\begin{aligned}
 &B\cdot R\wedge(A\cdot\rho)\\
 =&(\iota_Y+\eta\wedge)(R\wedge\iota_X(\rho)+R\wedge\xi\wedge\rho)\\
 =&\iota_Y\big(R\wedge\iota_X(\rho)\big)+\iota_Y(R\wedge\xi\wedge\rho)+\eta\wedge R\wedge\iota_X(\rho)+\eta\wedge R\wedge\xi\wedge\rho\\
 =&\quad\iota_Y(R)\wedge\iota_X(\rho)+(-1)^k R\wedge\iota_Y\iota_X(\rho)+\iota_Y(R)\wedge\xi\wedge\rho+(-1)^kR\wedge\iota_Y(\xi)\wedge\rho\\
  &+(-1)^{k-1}R\wedge\xi\wedge\iota_Y(\rho)+\eta\wedge R\wedge\iota_X(\rho)+\eta\wedge
  R\wedge\xi\wedge\rho
\end{aligned}
\end{equation}
and
\begin{equation}\label{BA}
B\cdot
A\cdot\rho=\iota_Y\iota_X(\rho)+\iota_Y(\xi)\wedge\rho-\xi\wedge\iota_Y(\rho)+\eta\wedge\iota_X(\rho)+\eta\wedge\xi\wedge\rho.
\end{equation}
Hence, we have
\begin{equation}\label{initial2}
\begin{aligned}
   &(A\cdot B)\cdot (R\wedge\rho)\\
 =&(\iota_X+\xi\wedge)\big(\iota_Y(R\wedge\rho)+\eta\wedge R\wedge\rho\big)\\
 =&\iota_X\iota_Y(R\wedge\rho)+\iota_X(\eta\wedge R\wedge\rho)+\xi\wedge\iota_Y(R\wedge\rho)+\xi\wedge\eta\wedge R\wedge\rho\\
 =&\quad \iota_X\iota_Y(R)\wedge\rho+(-1)^{k-1}\iota_Y(R)\wedge\iota_X(\rho)+(-1)^{k}\iota_X(R)\wedge\iota_Y(\rho)+R\wedge\iota_X\iota_Y(\rho)\\
  &+\iota_X(\eta)\wedge R\wedge\rho-\eta\wedge \iota_X(R)\wedge\rho+(-1)^{k+1}\eta\wedge R\wedge\iota_X(\rho)\\
  &+\xi\wedge\iota_Y(R)\wedge\rho+(-1)^{k}\xi\wedge
  R\wedge\iota_Y(\rho)+R\wedge\xi\wedge\eta\wedge\rho\\
 =&R\wedge (B\cdot A\cdot\rho)+B\cdot R\wedge (A\cdot\rho)-A\cdot R\wedge
 (B\cdot\rho)-\iota_Y\iota_X(R)\wedge\rho,
\end{aligned}
\end{equation}
where the last equality applies the equalities (\ref{BHA}) and
(\ref{BA}). So by combining (\ref{initial}) with the last term
dropped and (\ref{initial2}) with minus sign, we complete the proof
of the initial case of (\ref{Ht}).

Then, by induction on the degrees of $A$ and $B$, we can conclude
the proof. Actually, we just need to assume that (\ref{Ht}) holds
for all $p\leq r$ and $q=s$ and then show that it holds for $p= r+1$
and $q=s$ since (\ref{Ht}) is graded symmetric in $A$ and $B$. Here
we set $A=A_0\wedge \tilde{A}$ with $\tilde{A}=A_1\wedge\cdots\wedge
A_r$ and $B=B_1\wedge\cdots\wedge B_s$, where all $A_i,B_i\in
C^\infty(E^*)$. Assume that $A_0=X+\xi$, where $X\in C^\infty(T)$
and $\xi\in C^\infty(T^*)$. Then, one has
\begin{equation}\label{dH}
\begin{aligned}
   &d_R(A\cdot B\cdot\rho)\\
 =&(d-R\wedge)(\iota_X+\xi\wedge)(\tilde{A}\cdot B\cdot\rho)\\
 =&\big(L_X+d\xi\wedge\cdot-\iota_X(R)\wedge\cdot\big)(\tilde{A}\cdot B\cdot\rho)-A_0\cdot d_R(\tilde{A}\cdot B\cdot\rho)\\
 =&\big(L_X+d\xi\wedge\cdot-\iota_X(R)\wedge\cdot\big)(\tilde{A}\cdot B\cdot\rho)-A_0\cdot\big((-1)^r\tilde{A}\cdot d_R(B\cdot\rho)\\
 &+(-1)^{(r-1)s}B\cdot d_R(\tilde{A}\cdot\rho)+(-1)^{r-1}[\tilde{A},B]_R\cdot\rho+(-1)^{r+s+1}\tilde{A}\cdot B\cdot d_R\rho\big)\\
 =&(-1)^{(r+1)}A\cdot d_R(B\cdot\rho)+(-1)^{rs}B\cdot d_R({A}\cdot\rho)+(-1)^{r+s}A\cdot B\cdot d_R\rho+(-1)^{r}(A_0\wedge[\tilde{A},B]_R)\cdot\rho\\
 &+\big(L_X+d\xi\wedge\cdot-\iota_X(R)\wedge\cdot\big)(\tilde{A}\cdot B\cdot\rho)-(-1)^{rs}B\cdot\big(L_X+d\xi\wedge\cdot-\iota_X(R)\wedge\cdot\big)(\tilde{A}\cdot\rho).
\end{aligned}
\end{equation}

Next, we need the following

\begin{Claim} \label{Claim}For any $C\in C^\infty(E^*)$ and $\alpha\in
\Omega^*(\check{M},\mathbb{C})$, we have
$$[L_X+d\xi\wedge\cdot-\iota_X(R)\wedge\cdot,C\cdot]\alpha=[A_0,C]_R\cdot\alpha,$$
where the bracket $[\cdot,\cdot]$ on the left-hand side of the
equality is just the usual Lie bracket.
\end{Claim}
Before the proof, we can easily see from this claim that the last
two terms on the right-hand side of (\ref{dH}) combine to give us
$$(-1)^{rs}([A_0,B]_R\wedge\tilde{A})\cdot\rho$$
and then the last three terms  on the right-hand side of (\ref{dH})
combine to give
$$(-1)^{r}[A,B]_R\cdot\rho.$$
Hence, one has
$$ d_R(A\cdot B\cdot\rho)=(-1)^{(r+1)}A\cdot d_R(B\cdot\rho)+(-1)^{rs}B\cdot d_R({A}\cdot\rho)+(-1)^{r+s}A\cdot B\cdot d_R\rho
+(-1)^{r}[A,B]_R\cdot\rho,$$ by which we complete the induction.

Finally, we prove Claim \ref{Claim} to conclude the proof. If we
write $C=Y+\eta$, then
\begin{align*}
   &[L_X+d\xi\wedge\cdot-\iota_X(R)\wedge\cdot,C\cdot]\alpha\\
 =&L_X\iota_Y(\alpha)-\iota_Y(L_X\alpha)+(L_X\eta)\wedge\alpha+\iota_Y\iota_X(R)\wedge\alpha-\iota_Y(d\xi)\wedge\alpha\\
 =&\iota_{[X,Y]}(\alpha)+(L_X\eta)\wedge\alpha+\iota_Y\iota_X(R)\wedge\alpha-\iota_Y(d\xi)\wedge\alpha\\
 =&\iota_{[X,Y]}(\alpha)+(L_X\eta)\wedge\alpha-(L_Y\xi)\wedge\alpha-\frac{1}{2}d\big(\iota_X(\eta)-\iota_Y(\xi)\big)\wedge\alpha+\iota_Y\iota_X(R)\wedge\alpha\\
 =&[X+\xi,Y+\eta]\cdot\alpha+\iota_Y\iota_X(R)\wedge\alpha\\
 =&[A_0,C]_R\cdot\alpha,
\end{align*}
where the second equality uses the commutator formula (\ref{Ci}):
$L_X\circ\iota_Y-\iota_Y\circ L_X=\iota_{[X,Y]},$ and the third
equality uses the fact that $\iota_X(\eta)+\iota_Y(\xi)=0$ and the
classical Cartan formula $L_X=d\circ\iota_X+\iota_X\circ d$.
\end{proof}

As a direct corollary of Proposition \ref{Hodd}, one has

\begin{coro}\label{H1AB2}
For any smooth differential form $\rho$, any smooth $1$-form $R$ and
any $A,B\in C^\infty(\wedge^2 E^*)$, we have
\begin{equation}\label{iH1AB2}
d_R(A\cdot B\cdot\rho)=A\cdot d_R(B\cdot\rho)+B\cdot
d_R(A\cdot\rho)-[A,B]\cdot\rho-A\cdot B\cdot d_R\rho.
\end{equation}
\end{coro}
Obviously, similar to (\ref{gTT}) vs (\ref{TT}), we can obtain an
equivalent form of (\ref{iH1AB2}).

Then, based on the discussions above, we can reprove Corollary
\ref{TT2} for any Hermitian holomorphic vector bundle on a complex
manifold. Here we follow the notations in the previous section.

\begin{coro} Let $V$ be an arbitrary Hermitian holomorphic vector bundle on the complex manifold $M$.
For any $\omega\in A^{*,*}(V)$ and any $\phi_i\in
A^{0,1}(M,T^{1,0}_{M})$, $i=1,2$, there holds
$$[\phi_1,\phi_2]\lrcorner\omega=\phi_1\lrcorner\nabla(\phi_2\lrcorner\omega)-\nabla\big(\phi_2\lrcorner(\phi_1\lrcorner\omega)\big)
+\phi_2\lrcorner\nabla(\phi_1\lrcorner\omega)-\phi_2\lrcorner(\phi_1\lrcorner\nabla\omega),$$
where $\nabla$ is the Chern connection of the Hermitian holomorphic
vector bundle $V$.
\end{coro}
\begin{proof}
This corollary is a direct application of Corollary \ref{H1AB2} when
we set $A=(\phi_1)^i\cdot\partial_i$ and
$B=(\phi_2)^j\cdot\partial_j$ and take $R$ as the connection
$(1,0)$-form matrix $\theta$ of the connection $\nabla$ with respect
to a holomorphic frame $s$ of $V$ with minus sign, by the same
principle as we choose $g\cdot\partial f$ in the proof of Theorem
\ref{TTh}. It is obvious that $E$ in Corollary \ref{H1AB2} is taken
as $T^{0,1}\otimes {T^*}^{1,0}$ in our case. More precisely, since
\begin{align*}
   [A,B]
 =&[(\phi_1)^i\cdot\partial_i,(\phi_2)^j\cdot\partial_j]\\
 =&-[(\phi_1)^i,\partial_j]\wedge\partial_i\wedge(\phi_2)^j-[\partial_i,(\phi_2)^j]\wedge(\phi_1)^i\wedge\partial_j\\
 =&\iota_{\partial_j}\big(d(\phi_1)^i\big)\wedge\partial_i\wedge(\phi_2)^j-\iota_{\partial_i}\big(d(\phi_1)^j\big)\wedge(\phi_1)^i\wedge\partial_j\\
 =&(\phi_2)^j\wedge\partial_j(\phi_1)^i\wedge\partial_i+(\phi_1)^i\wedge\partial_j(\phi_2)^j\wedge\partial_j,
\end{align*}
then one has $$[A,B]\cdot\omega=[\phi_1,\phi_2]\lrcorner\omega.$$
Moreover, one easily knows that
$$A\cdot d_R(B\cdot\omega)=\phi_1\lrcorner\nabla(\phi_2\lrcorner\omega),$$
$$B\cdot d_R(A\cdot\omega)=\phi_2\lrcorner\nabla(\phi_1\lrcorner\omega),$$
$$d_R(A\cdot B\cdot\omega)=\nabla(\phi_2\lrcorner(\phi_1\lrcorner\omega))$$
and
$$A\cdot B\cdot d_R\omega=\phi_2\lrcorner(\phi_1\lrcorner\nabla\omega).$$

Hence, by substituting the five equalities above into
$(\ref{iH1AB2})\otimes s$ the equivalent form of (\ref{iH1AB2}), we
complete our proof.
\end{proof}

Almost by the same argument as the previous corollary, we can
generalize it to any polyvector fields as follows.
\begin{coro}\label{main} Let $V$ be an arbitrary Hermitian holomorphic vector bundle on the complex manifold $M$.
For any $\omega\in A^{*,*}(V)$ and any $\phi_i\in
A^{0,q_i}(M,\bigwedge^{p_i}T^{1,0}_{M})$, $i=1,2$, there holds
$$[\phi_1,\phi_2]\lrcorner\omega=\phi_1\lrcorner\nabla(\phi_2\lrcorner\omega)-\nabla\big(\phi_2\lrcorner(\phi_1\lrcorner\omega)\big)
+\phi_2\lrcorner\nabla(\phi_1\lrcorner\omega)-\phi_2\lrcorner(\phi_1\lrcorner\nabla\omega),$$
where $\nabla$ is the Chern connection of the Hermitian holomorphic
vector bundle $V$ and $[\cdot,\cdot]$ on the LHS is the standard
Schouten-Nijenhuys bracket on polyvector fields. For convention,
here we assume that the first bidegree of $\omega$ is not less than
any $p_i$.
\end{coro}

\section{Acknowledgement} The second author would like to
thank Professor A. Todorov, Dr. Fangliang Yin and Shengmao Zhu for
several useful talks, and specially Dr. Feng Guan for his
constructive suggestion on Theorem \ref{TTh} during his visit to
Center of Mathematical Sciences, Zhejiang University. Moreover, we
also would like to thank the referee for their useful and accurate
comments improving our presentation a lot.  \vspace{1cm}

%\addcontentsline{toc}{section}{} \vspace{1cm}

\end{document}